\documentclass[10pt,reqno]{amsart}
\usepackage{hyperref}
\usepackage{amsmath,amssymb}
\usepackage[english]{babel}
\usepackage{caption}
\usepackage{amssymb,amsmath,euscript,enumerate,tikz}
\usepackage[margin=1in]{geometry}
\usepackage{graphicx}
\usepackage{float}

\usepackage{pgf,tikz}
\usepackage{mathrsfs}
\usepackage{upgreek}
\usepackage{mathtools}

\usepackage{mathtools}

\theoremstyle{plain}
\newtheorem{thm}{Theorem}[section]
\newtheorem{lem}[thm]{Lemma}
\newtheorem{prop}[thm]{Proposition}
\newtheorem{cor}[thm]{Corollary}
\newtheorem{rem}[thm]{Remark}
\newtheorem{defn}[thm]{Definition}
\newtheorem{exmp}[thm]{Example}

\def\m{\mathfrak{m}}
\def\k{\mathbb{K}}

\newcommand \reg{\operatorname{reg}}
\newcommand \Tor{\operatorname{Tor}}
\newcommand \lcm{\operatorname{lcm}}
\newcommand \cha{\operatorname{char}}

\newcommand \pd{\operatorname{pd}}

\newcommand \ass{\operatorname{Ass}}
\newcommand \Min{\operatorname{Min}}
\newcommand \depth{\operatorname{depth}}

\newcommand \p{\mathfrak{p}}
\newcommand \K{\mathbb{K}}

\begin{document}
	\title{Properties of analogues of Frobenius powers of ideals}
	\author{Subhajit Chanda}
	\email{subhajitchandamath@gmail.com}
	\author{Arvind Kumar}
	\email{arvkumar11@gmail.com}
	\address{Department of Mathematics \\
		New Academic Complex, IIT Madras \\ Chennai \\
		Tamil Nadu, India-600036}

	\begin{abstract}
		Let $R=\K[X_1, \hdots , X_n ]$ be a polynomial ring over a field $\K$. We introduce an endomorphism $\mathcal{F}^{[m]}: R \rightarrow R $ and denote the image of an ideal $I$ of $R$ via this endomorphism as $I^{[m]}$ and call it to be the $m$\textit{-th square power} of $I$. In this article, we study some homological invariants of $I^{[m]}$ such as regularity, projective dimension, associated primes and depth for some families of ideals e.g. monomial ideals. 
	\end{abstract}
	\maketitle

	\section{Introduction}
	Let $R$ be a commutative Noetherian ring with unity. When  $\cha(R)$ is a prime number $p$, then for any positive integer $e$, one can consider Frobenius endomorphism $F^e : R\rightarrow R$ which is defined as $F^e(r)= r^{p^e}$ for $r \in R$. In this way, $R$ can be viewed as an $R$-module via $F^e$, and this  module is denoted by $R^e$.  For an ideal $I \subset R$ and $e \geq 1$, $I^{[p^e]}$ denotes the ideal generated by the $p^e$-th powers of elements of $I$ and it is called $p^e$\textit{-th Frobenius power} of $I$. From now, we fix $R=\K[X_1,\ldots,X_n]$ to be a polynomial ring over a field $\K$. Whenever we say $I \subset R$ is a homogeneous ideal, then we mean that $R$ is a standard graded polynomial ring and $I$ is an ideal generated by homogeneous elements in $R$. When $\cha(\K)=p,$
	Katzman \cite{katzman} conjectured that if $I \subset R$ and $J \subset R$ are homogeneous ideals, then $\reg( R/ (J+I^{[p^e]})) $ grow linearly in $p^e$.  Katzman \cite[Corollary 23]{katzman} proved this conjecture for certain classes of ideals (e.g., $I$ a homogeneous ideal and $J$ a monomial ideal). The behavior of $\reg(R/I^{[p^e]})$ has been of great interest throughout the last few decades, see \cite{brenner, chardin,  stefani}. Moreover, the growth of $\reg(R/I^{[p^e]})$ is of independent interest as this is connected to the localization of tight closure \cite{katzman}. In this article, we  give a precise formula for $\reg(R/I^{[p^e]})$ in terms of $\reg(R/I)$ for any homogeneous ideal $I \subset R$.  We study this problem in a more general setup as follows.    
	
	Let $R=\K[X_1,\ldots,X_n]$ be a polynomial ring over a field $\K$ (need not be of prime characteristic).  For any positive integer $m$, we define an endomorphism $\mathcal{F}^{[m]} : R\rightarrow R$ as $\mathcal{F}^{[m]}(f)= f^{[m]}$ for $f \in R$, where $f^{[m]}= f(X_1^m, \hdots , X_n^m)$. In this way, $R$ can be viewed as an $R$-module via $\mathcal{F}^{[m]}$. Let $I \subset R$ be an ideal (need not be homogeneous). We define $I^{[m]} := \langle f^{[m]} : f \in I \rangle $, and call it to be the $m$\textit{-th square power} of $I$.   { This notion first appears in \cite{nvv}.} If $I$ is a monomial ideal or $\K$ is the finite field of order $p$, then the $p^e$-th Frobenius power of $I$ is same as the $p^e$-th square power of $I$. Whenever we write $I^{[p^e]}$, we mean that $\cha(\K) =p$ is prime and $I^{[p^e]}$ is the $p^e$-th Frobenius power of $I$, otherwise the base field is of arbitrary characteristic  and $I^{[m]} $ is the $m$-th square power of $I$.    
	{In \cite{nvv}, Neves,  Vaz Pinto and Villarreal obtained the regularity of $\reg(R/I^{[m]})$ when $R/I$ is Cohen-Macaulay. We generalize their result for homogeneous ideals.}   
	We prove that $\reg(R/I^{[m]})$ is a linear function in $m$ by obtaining a precise formula for $\reg(R/I^{[m]})$ in terms of $\reg(R/I)$ (Corollary \ref{reg}). The proof relies on the relation between the graded Betti numbers of $R/I^{[m]}$ and $R/I$ (Corollary \ref{main-thm2}), and the relation between their extremal Betti numbers (Proposition \ref{extremal}).
	We also prove analogous statements for Frobenius powers of homogeneous ideals (Theorem \ref{frob-thm}, Proposition \ref{extremal-frobenius}).
	
	We also study the behavior of ${\pd}(R/I^{[m]})$ and ${\ass}(I^{[m]})$ for $m\geq 2$. The following well-known lemma relates projective dimension and the sets of associated primes of $R/I$ and $R/I^{[p^e]}$. 
	\begin{lem}[Peskine-Szpiro]\cite[Lemma 2.2]{hochster}\label{peskine}
		Let $I \subset R$ be an ideal. Then,  
		\begin{itemize}
			\item[(a)] ${\pd} (R/I^{[p^e]}) = {\pd} (R/I)$ for all $e\geq 1$.
			\item[(b)] $\ass(R/I^{[p^e]})=\ass(R/I)$ for all $e \geq 1$.
		\end{itemize}
	\end{lem}
	We prove the analogue of statement(a) (Corollary ~\ref{cor-depth}) for projective dimension of square powers.
	The analogous statement for associated primes does not hold in general, see Example ~\ref{count ass}. For monomial ideals,  we prove $\ass(R/I^{[m]})=\ass(R/I)$ for all $m \geq 1$ (Theorem ~\ref{main-thm1}).  
	
	The study of regularity of    powers of homogeneous ideals has been a central problem in commutative algebra and algebraic geometry.     Cutkosky, Herzog and Trung \cite{CHT}, and independently  Kodiyalam \cite{vijay} proved that if $I$ is a homogeneous ideal of $R$, then  $\reg(R/I^s)$ is a linear function in $s$,     i.e., there exist non-negative integers $a$ and $b$ depending on $I$ such that $\reg(R/I^s)=as+b \text{ for all $s \gg 0$}.$
	Using square powers of  homogeneous ideal $I$, we prove that there exist $1\leq i \leq n$ such that $\beta_{i,i+\reg(R/I^s)}(R/I^s)$ is an extremal Betti number for $s \gg 0$ (Theorem \ref{extremal-power}).

	We then move on to study square powers of monomial ideals. First, we compute the primary decomposition of square powers of monomial ideals (Theorem \ref{main-thm1}).   Then, we study normally torsion-freeness of square powers of monomial ideals. We prove that for a monomial ideal $I$, $I^{[m]}$ is normally torsion-free if and only if $I$ is normally torsion-free (Theorem \ref{normally torsion free}). Gitler, Valencia and Villarreal \cite[Corollary 3.14]{gvv} classified   {all} normally torsion-free square-free monomial ideals. Moreover,  they gave an effective method using optimization to determine when a square-free monomial ideal is normally torsion-free. Along with their results Theorem \ref{normally torsion free} gives an effective criterion to determine whether $I^{[m]}$ is normally torsion-free. Brodmann \cite{brod} proved that the sequence $\lbrace \ass({R}/{I^s}) \rbrace _ {s\geq 1}$ of associated prime ideals is stationary for $s\gg0$. The least positive integer $s_0 $ such that $\ass({R}/{I^s})= \ass({R}/{I^{s_0}})$ for all $s\geq s_0 $  is called the \textit{index of stability} of $I$. We prove that the index of stability of monomial ideal $I$ and $I^{[m]}$ are the same (Theorem ~\ref{index}).
	
	Next, we study the depth of symbolic powers of  monomial ideals. Nguyen and Trung proved that ${\depth}({R}/{I^{(s)}})$ is  periodic for $s\gg0$, \cite[Proposition 3.1]{TrungInv}. Moreover, they  proved that for an extensive class of monomial ideals (including all square-free monomial ideals) ${\depth}({R}/{I^{(s)}})$ is always a convergent function. They proved:
	\begin{thm}\cite[Theorem 3.1]{TrungInv}
		Let $I$ be a monomial ideal in $R$ such that $I^{(s)}$ is integrally closed for $s\gg0$. Then,  ${\depth} ({R}/{I^{(s)}})$ is a convergent function with 
		\begin{displaymath}
			\lim_{s\to \infty} {\depth} ({R}/{I^{(s)}})= \dim (R) - \dim ( F_{\mathfrak{s}}(I)),
		\end{displaymath}
		which is also the minimum of $\depth({R}/{I^{(s)}})$ among all integrally closed symbolic powers $I^{(s)}$.
	\end{thm}
	In  Lemma \ref{integral}, we prove that $(I^{[m]})^{(s)}$ is not integrally closed for $m \geq 2, s\geq 2$. If $I$ is a monomial  ideal such that $I^{(s)}$ is integrally closed for $s \gg 0$, then for all $m \geq 2$,  the sequence $\{\depth(R/(I^{[m]})^{(s)})\}_{s\geq1}$ is a convergent sequence  (Theorem ~\ref{main-thm3}) and $$ \lim_{s \rightarrow \infty} \depth(R/(I^{[m]})^{(s)})=\dim (R)- \dim(F_{\mathfrak{s}}(I^{[m]})).$$

	\section{Square Powers of an Ideal}
	Throughout the article, we fix $\k$ to be any field, and $R=\K[X_1,\ldots,X_n]$ to be a polynomial ring. We assume that $0 \in \mathbb{N}$. For $\alpha =(\alpha_1,\ldots,\alpha_n) \in \mathbb{N}^n$, set $X^{\alpha}=\prod_{i=1}^n X_i^{\alpha_i}.$ Set $\widetilde{R}=\K[X_1^m,\ldots,X_n^m]$ for any positive integer $m$.
	\begin{defn}\label{main defn}
		Let $f \in R$ be an element. We set $f^{[m]} =f(X_1^m,\ldots,X_n^m)$.    Let $I$ be an ideal in $R$. Then,  for any positive integer $m$, we define $I^{[m]}:=\langle f^{[m]} \; : \; f \in I \rangle$. We call $I^{[m]}$ to be the $m$-th square power of $I$.
	\end{defn}
	The following remark forms the basis for the proofs of this article.
	\begin{rem}\label{rmk}
		The square power operation preserves products and sums, i.e.,
		$(f + g)^{[m]} = (f)^{[m]} + (g)^{[m]}$ and
		$(fg)^{[m]} = (f)^{[m]} (g)^{[m]}$ and hence defines a ring isomorphism from $R$ to $\widetilde{R}$. Thus, $I^{[m]} $ is the extension of the ideal $( \mathcal{F}^{[m]}(I) )$.
	\end{rem}
	
	\begin{rem}\label{frob power}
		If $\k$ is the finite field of order $p$, and $I$ is a homogeneous ideal in $R$, then ${p^e}$-th square power of $I$ coincides with the $p^e$-th Frobenius power of $I$. 
	\end{rem}
	
	We first prove the following result that is needed for the rest of the article.
	\begin{lem}
		Let $I$ be an ideal of $R$ and $f \in R$. Then, $f \in I$ if and only if $f^{[m]} \in I^{[m]}$.
	\end{lem}
	\begin{proof}
		Let $I$ be generated by polynomials $f_1,\ldots,f_k$.  We claim that $I^{[m]} = \langle f_1^{[m]},\ldots, f_k^{[m]} \rangle$. Clearly, $\langle f_1^{[m]},\ldots, f_k^{[m]} \rangle \subset I^{[m]}$. Now, let $g \in I^{[m]}$ be any element. Then, $g = \sum_{i=1}^r h_ig_i^{[m]}$ for some $h_i \in R$ and $g_i \in I$. Now, for each $i$,  $g_i =\sum_{j=1}^{k} h'_{ij} f_j$ for some $h'_{ij} \in R$. Therefore, $g_i^{[m]} =\sum_{j=1}^{k} {h'}_{ij}^{[m]} f_j^{[m]}$. Thus, $g = \sum_{i=1}^r \sum_{j=1}^k h_i{h'}_{ij}^{[m]} f_j^{[m]} \in \langle f_1^{[m]},\ldots, f_k^{[m]} \rangle$. Hence, $I^{[m]} = \langle f_1^{[m]},\ldots, f_k^{[m]} \rangle$.
		
		Let $f \in R$ be such that $f^{[m]} \in I^{[m]}$. Then, $f^{[m]}  \in I^{[m]} \cap \widetilde{R}$, and hence, $f^{[m]}= \sum_{i=1}^k h_i^{[m]}f_i^{[m]}=(\sum_{i=1}^k h_if_i)^{[m]}$ for some $h_i \in R$. Therefore, $f= \sum_{i=1}^k h_if_i \in I$. Hence, the assertion follows.
	\end{proof}
	
	{In \cite{nvv}, Neves, Vaz Pinto and Villarreal proved that if $I$ is a complete intersection graded lattice ideal, then $I^{[m]}$ is also complete intersection ideal, see \cite[Section 4]{nvv}. We generalize their result for graded ideals.}
	
	\begin{prop}\label{m-regular}
		If $f_1,\ldots,f_r \in R$ is an $R$-regular sequence, then, for any $m \geq 2$, $f_1^{[m]},\ldots,f_r^{[m]}$ is an $R$-regular sequence.
	\end{prop}
	\begin{proof} Let $m \geq 2$.  Observe that $\{X^a : \; a \in \mathbb{N}^n, \text{ and } a_i <m\}$ is a free basis of $\widetilde{R}$-module $R$.   By \cite[Proposition 1.6.7]{Her1}, the Koszul homology of $f_1^{[m]},\ldots,f_r^{[m]}$ is   $H_{\bullet}(f_1^{[m]},\ldots,f_r^{[m]};R)=H_{\bullet}(f_1^{[m]},\ldots,f_r^{[m]};\widetilde{R}) \otimes_{\widetilde{R}}R$. We define $\varPhi : R \rightarrow \widetilde{R}$ as $\varPhi(X_i)= X_i^m.$ Observe that $\varPhi$ is a ring isomorphism.  Thus, $f_1^{[m]},\ldots,f_r^{[m]}$ is an $\widetilde{R}$-regular sequence. Now, by \cite[Corollary 1.6.14]{Her1},  $f_1^{[m]},\ldots,f_r^{[m]}$ is an $R$-regular sequence.             
	\end{proof}

	Let $I\subset R $ be a homogeneous ideal. We aim to study algebraic invariants associated with the minimal free resolution of $R /I^{[m]}$. First, we recall some facts about the minimal free resolution.
	
	Let	\[(\mathfrak{F}_{\bullet}, \phi_{\bullet}) \; \; : \; \; 0 \longrightarrow \mathfrak{F}_{t} \xrightarrow{\phi_t} \mathfrak{F}_{t-1} \xrightarrow{\phi_{t-1}} \cdots \xrightarrow{\phi_2} \mathfrak{F}_1 \xrightarrow{\phi_1} \mathfrak{F}_0 \longrightarrow R /I \longrightarrow 0 \] be the minimal graded free resolution of $R /I$, 
	where $\mathfrak{F}_i=\bigoplus_kR (-k)^{\beta_{i,k}(R /I)}$ for $i\geq 1$, $\mathfrak{F}_0=R$ and the matrix $\phi_1$ is given by minimal homoogeneous set of generators. Here, $t$ is  the \textit{projective dimension} of $R /I$, and it is denoted by $\pd(R /I)$. Also, $R (-k)$ is the free $R $-module of rank $1$ generated in degree $k$. 
	The number $\beta_{i,k}(R /I)$ is called the $(i,k)$-th \textit{graded Betti number} of $R /I$, and this is uniquely determined by $I$, i.e., $\beta_{i,k}(R /I) = \dim_\K \Tor_i^R (\K,R /I)_{k},  \text{ for } i \geq 0, \text{ and } k \in \mathbb{Z}$.      
	For more details on the graded minimal free resolution, we refer the readers to \cite{Her1} and \cite{peeva}.
	
	{In \cite{nvv}, Neves, Vaz Pinto and Villarreal computed the graded minimal free resolution of ${R} /I^{[m]}$. They proved: 		
		\begin{thm}\cite[Lemma 3.7 (a)]{nvv}(or \cite[Lemma 3.1]{GHM})\label{main-thm-free}
			Let $I\subset R $ be a homogeneous  ideal and $m \geq 2$. If $(\mathfrak{F}_{\bullet}, \phi_{\bullet})$ is the graded minimal free  resolution of $R /I$, then $(\mathfrak{F}_{\bullet}^{[m]}, \phi_{\bullet}^{[m]})$ is the graded minimal free  resolution of $R /I^{[m]}$, where for all $i \geq 1$, $$\mathfrak{F}_i^{[m]}=\bigoplus_kR (-mk)^{\beta_{i,k}(R /I)},$$ and the matrix $\phi_{i}^{[m]}$ is given by $$(\phi_i^{[m]})_{k,l}= ((\phi_i)_{k,l})^{[m]}.$$ 
	\end{thm}}
	
	As an immediate consequence, we compute graded Betti numbers of $R /I^{[m]}$ in terms of graded Betti numbers of $R /I$. 
	\begin{cor}\label{main-thm2}
		Let $I\subset R $ be a homogeneous  ideal and $m \geq 2$. Then,  $\beta_{i,j}(R /I^{[m]})$ is a non-zero graded Betti number  only if $j=mk$ for some $k$, and $\beta_{i,k}(R /I)$ is a non-zero graded Betti number. Moreover, $\beta_{i,mk}(R /I^{[m]})=\beta_{i,k}(R /I)$, for every $k,$ and $i \geq 0$.
	\end{cor}
	\begin{proof}
		The assertion follows from Theorem \ref{main-thm-free}.
	\end{proof}
	We can conclude the analogue of statement(a) of Lemma \ref{peskine} for projective dimension of square powers as
	an immediate consequence of Theorem \ref{main-thm-free}.
	\begin{cor}\label{cor-depth}
		Let $I\subset R$ be a homogeneous ideal and $m \geq 2$. Then,  $\pd(R/I)=\pd(R/I^{[m]})$, for all $m \geq 1$. 
	\end{cor}
	
	{The following result is an  immediate consequence of Corollary \ref{cor-depth}.
		\begin{cor}\label{cor-depth1}
			Let $I\subset R$ be a monomial ideal and $m \geq 2$. If $R/I$ is Cohen-Macaulay (Gorenstein), then $R/I^{[m]}$ is Cohen-Macaulay (Gorenstein).
		\end{cor}
		\begin{proof}
			By Auslander-Buchsbaum formula and Corollary \ref{cor-depth}, we get that $\depth(R/I)=\depth(R/I^{[m]})$, for all $m \geq 2$. Note that $\varPhi : R \rightarrow \widetilde{R}$ defined  as $\varPhi(X_i)= X_i^m$ is a ring isomorphism.  Therefore, $\dim(R/I)=\dim(\widetilde{R}/\varPhi(I))$. Observe that $\widetilde{R}\subset R$ is integral extension. Thus, $\dim(R/I)=\dim({R}/I^{[m]})$.  Hence, if $R/I$ is Cohen-Macaulay,  then $R/I^{[m]}$ is Cohen-Macaulay. Now, assume that $R/I$ is Gorenstein. Thus, by \cite[Theorem 3.2.10]{Her1} $R/I$ is Cohen-Macaulay of type $1$. Hence,  the assertion follows from Corollary \ref{main-thm2}.
	\end{proof} }   
	We now compute graded Betti numbers of $p^e$-th Frobenius power of $I$  in terms of graded Betti numbers of $I$.
	\begin{thm}\label{frob-thm}
		Assume that $\cha(\K) =p>0$.    Let $I\subset R $ be a homogeneous  ideal. Then,  $\beta_{i,j}(R /I^{[p^e]})$ is a non-zero graded Betti number  only if $j=p^ek$ for some $k$, and $\beta_{i,k}(R /I)$ is a non-zero graded Betti number. Moreover, $\beta_{i,p^ek}(R /I^{[p^e]})=\beta_{i,k}(R /I)$, for every $k,$ and $i \geq 0$.
	\end{thm}
	\begin{proof} 
		{
			Let $(\mathfrak{F}_{\bullet}, \phi_{\bullet})$ be the graded minimal free  resolution of $R/I$. For $ i\geq 0$, set $$\mathfrak{F}_i^{[p^e]}=\bigoplus_kR(-p^ek)^{\beta_{i,k}(R/I)},$$ and the matrix $\phi_{i}^{[p^e]}$ given by $$(\phi_i^{[p^e]})_{k,l}= ((\phi_i)_{k,l})^{p^e}.$$  For each $ i\geq 1$, $\phi_i^{[p^e]} \circ \phi_{i+1}^{[p^e]} =(\phi_i \circ \phi_{i+1})^{[p^e]}=0$.   Now, we prove that  $(\mathfrak{F}_{\bullet}^{[p^e]}, \phi_{\bullet}^{[p^e]})$ is a free resolution of $R/I^{[p^e]}$. The proof is similar to the proof of \cite[Theorem 8.2.7]{Her1}. We prove here for sake of completeness.   Set $r_i= \sum_{j=i}^{t} (-1)^{j-i}\text{rank}(\mathfrak{F}_j)$.  Since $(\mathfrak{F}_{\bullet}, \phi_{\bullet})$ is a free resolution,  $\text{grade}\; I_{r_i}(\phi_i) \geq i$ for $i \geq 1$, where $I_{r_i}(\phi_i) $ is an ideal of $R $ generated by $r_i \times r_i $ minors of $\phi_i$. Note that  $r_i= \sum_{j=i}^{t} (-i)^{j-i}\text{rank}(\mathfrak{F}_j^{[p^e]})$, and $(I_{r_i}(\phi_i))^{[p^e]}=I_{r_i}(\phi_i^{[p^e]}) $. Since $(I_{r_i}(\phi_i))^{[p^e]}$  and $I_{r_i}(\phi_i) $ have same radical,   for $i \geq 1$,  $\text{grade}\; I_{r_i}(\phi_i^{[p^e]}) \geq i$. Thus, by \cite[Theorem 1.4.13]{Her1}, $(\mathfrak{F}_{\bullet}^{[p^e]}, \phi_{\bullet}^{[p^e]})$ is a free resolution of $R/I^{[p^e]}$. Note that for each $i \geq 1$, $\phi_i^{[p^e]}$ is a graded homomorphism of degree $0$. Therefore,  $(\mathfrak{F}_{\bullet}^{[p^e]}, \phi_{\bullet}^{[p^e]})$ is a graded free resolution of $R/I^{[p^e]}$. Hence, the assertion follows from \cite[Theorem 1.3.1]{Her1}.}
	\end{proof}

	One important algebraic invariant which is associated with the minimal free resolution of finitely generated graded module is the    Castelnuovo-Mumford regularity, henceforth called regularity. The 
	regularity of $R/I$, denoted by $\reg(R/I)$, is defined as 
	\[
	\reg(R/I):=\max \{j -i: \beta_{i,j}(R/I) \neq 0\}.
	\] A nonzero graded Betti number $\beta_{i,j}(R/I)$ is called an \textit{extremal Betti number}, if $\beta_{r,s}(R/I)=0$ 
	for all pairs $(r,s)\neq (i,j)$ with $r\geq i$ and $s\geq j$. Observe that $R/I$ admits a unique extremal Betti number if and only if 
	$\beta_{p,p+r}(R/I)\neq 0$, where $p =\pd(R/I)$ and $r=\reg(R/I)$.

	Now, we  study extremal Betti numbers  of $R/I^{[m]}$ in terms of extremal Betti numbers of $R/I$.
	\begin{prop}\label{extremal}
		Let $I$ be a homogeneous ideal of $R$, and  $m \geq 2$. If $\beta_{i,j}(R/I^{[m]})$ is an extremal Betti number, then  $j=ml$ for some $l$, and $\beta_{i,l}(R/I)$ is an extremal Betti number.Conversely, if $\beta_{i,j}(R/I)$ is an extremal Betti number, then $\beta_{i,mj}(R/I^{[m]})$ is an extremal Betti number.
	\end{prop}
	\begin{proof}
		Since $\beta_{i,j}(R/I^{[m]})$ is an extremal Betti number, $\beta_{i,j}(R/I^{[m]}) \neq 0$ and $\beta_{k,l}(R/I^{[m]}) =0$ for $(k,l) \neq (i,j)$ and $k \geq i, l \geq j$. By Corollary \ref{main-thm2}, $\beta_{i,j}(R/I^{[m]})=\beta_{i,j/m}(R/I) \neq 0$. Let $(k,l) \neq (i,j/m)$ such that $k \geq i, l \geq j/m$. Then, by Corollary \ref{main-thm2}, $\beta_{k,l}(R/I)=\beta_{k,lm}(R/I^{[m]}) =0$ as $k \geq i$ and $lm \geq j$. Hence, $\beta_{i,j/m}(R/I)$ is an extremal Betti number.
		
		Now, assume that $\beta_{i,j}(R/I)$ is an extremal Betti number.  By Corollary \ref{main-thm2}, $\beta_{i,mj}(R/I^{[m]})=\beta_{i,j}(R/I) \neq 0$.  Let $(k,l) \neq (i,mj)$ such that $k \geq i$ and $l \geq mj$. If $l$ is not multiple of $m$, then, by Corollary \ref{main-thm2}, $\beta_{k,l}(R/I^{[m]})=0$. Thus, we may assume that $l=ml'$, for some $l'$. Then, $\beta_{k,l}(R/I^{[m]})=\beta_{k,l'}(R/I)=0$ as $k \geq i$ and $l' \geq j$. Hence, $\beta_{i,mj}(R/I^{[m]})$ is an extremal Betti number.
	\end{proof}
	The above result also holds for Frobenius powers of homogeneous ideals. 
	\begin{prop}\label{extremal-frobenius}
		Assume that $\cha(\K) =p>0$. Let $I$ be a homogeneous ideal of $R$. If  $\beta_{i,j}(R/I^{[p^e]})$ is an extremal Betti number, then  $j=p^el$ for some $l$, and $\beta_{i,l}(R/I)$ is an extremal Betti number. Conversely, if $\beta_{i,j}(R/I)$ is an extremal Betti number, then $\beta_{i,p^ej}(R/I^{[p^e]})$ is an extremal Betti number.
	\end{prop}
	\begin{proof} The proof is similar to the proof of Proposition \ref{extremal}.
	\end{proof}
	We now compute the regularity of square power of homogeneous ideals and further show that it matches with the regularity of Frobenius power of the same when the characteristic of the underlying field is prime. {This result is a refinement of \cite[Lemma 3.7 (b)]{nvv}, where $R/I$ is assumed to be Cohen-Macaulay.}
	\begin{cor}\label{reg}
		Let $I$ be a homogeneous ideal of $R$. Let $r=\reg(R/I)$ and $i$ be the unique non-negative integer such that  $\beta_{i,i+r}(R/I)$ is an extremal Betti number. Then,  for any $m \geq 2$,  $\reg(R/I^{[m]})=mr+(m-1)i.$  Furthermore, if $\cha(\K)=p >0$, then $\reg(R/I^{[p^e]})=p^er+(p^e-1)i$, where $I^{[p^e]}$ denotes the Frobenius power of the ideal $I$.
	\end{cor}
	
	\begin{proof}
		It follows from Proposition \ref{extremal} that $\beta_{i,m(i+r)}(R/I^{[m]})$ is an extremal Betti number, and hence, $\beta_{i,m(i+r)}(R/I^{[m]}) \neq 0.$ Therefore, $\reg(R/I^{[m]} )\geq m(r+i)-i=mr+(m-1)i.$
		Set $r'=\reg(R/I^{[m]} )$. Then,  for some $k \leq i$, $\beta_{k,k+r'}(R/I^{[m]})$ is an extremal Betti number. Thus, by Proposition \ref{extremal},
		\begin{align*}\beta_{k,k+r'}(R/I^{[m]})= \beta_{k,(k+r')/m}(R/I) \neq 0.\end{align*} 
		Therefore, $r \geq (k+r')/m -k$, and hence, $r' \leq rm+(m-1)k \leq rm+(m-1)i$ as $ k \leq i$. Hence, $\reg(R/I^{[m]})=mr+(m-1)i.$
		The second assertion follows similarly by using Proposition ~\ref{extremal-frobenius}.
	\end{proof}
	
	We illustrate Corollaries \ref{main-thm2} and \ref{reg} with the following  example.
	\begin{exmp}
		Let $R=\k[X_1,X_2,X_3,X_4]$ and $I=(X_1X_2,X_1X_3,X_1X_4,X_2X_3,X_2X_4,X_3X_4).$ The following is the Betti diagram of $R/I$:
		
		\begin{center}
			$\begin{matrix}
			&0&1&2&3\\\text{total:}&1&6&8&3\\
			\text{0:}&1&\text{.}&\text{.}&\text{.}\\
			\text{1:}&\text{.}&6&8&3\\\end{matrix}$        
		\end{center}    
		The Betti diagram of $R/I^{[2]}$ is
		
		\begin{center}
			$\begin{matrix}
			&0&1&2&3\\\text{total:}&1&6&8&3\\
			\text{0:}&1&\text{.}&\text{.}&\text{.}\\
			\text{1:}&\text{.}&\text{.}&\text{.}&\text{.}\\
			\text{2:}&\text{.}&\text{.}&\text{.}&\text{.}\\
			\text{3:}&\text{.}&6&\text{.}&\text{.}\\\text{4:}&\text{.}&\text{.}&8&\text{.}\\
			\text{5:}&\text{.}&\text{.}&\text{.}&3\\\end{matrix}$
		\end{center}    
		One can note that $R/I$ has linear resolution, but the resolution of $R/I^{[2]}$ is not linear.    
	\end{exmp}    
	Now, we prove an auxiliary result.
	\begin{lem}\label{tech-lem11}
		Let $I\subset R$ be a homogeneous ideal. Then,  for all $m\geq 2 $ and $s\geq 2 $,
		$$ (I^{[m]})^s= (I^s)^{[m]}.$$
	\end{lem}
	\begin{proof} Let $\lbrace f_1, \ldots, f_r  \rbrace $ be a  minimal generating set of $I$. Then, $I^{[m]} = \langle f_1^{[m]},\ldots,f_r^{[m]} \rangle$. Also, $\lbrace f_1^{a_1}\ldots f_r^{a_r} \; : \; a_i \in \mathbb{N}, \;\sum_{i=1}^r a_i =s  \rbrace $ generates $I^s$. Note that $(f_1^{a_1}\ldots f_r^{a_r})^{[m]} =(f_1^{[m]})^{a_1}\ldots (f_r^{[m]})^{a_r}$ for all  $a_1, \ldots,a_r \in \mathbb{N}$. Therefore, $(I^s)^{[m]} = \langle (f_1^{a_1}\ldots f_r^{a_r})^{[m]} \; : \; a_i \in \mathbb{N}, \;\sum_{i=1}^r a_i =s  \rangle = \langle (f_1^{[m]})^{a_1}\ldots (f_r^{[m]})^{a_r} \; : \; a_i \in \mathbb{N}, \;\sum_{i=1}^r a_i =s  \rangle = (I^{[m]})^s$. Hence, the assertion follows.
	\end{proof}
	We use square powers of homogeneous ideals to prove the following result:
	\begin{thm}\label{extremal-power}
		Let $I$ be a homogeneous ideal. Then, there exists $i$ such that $\beta_{i,i+\reg(R/I^s)}(R/I^s)$ is an extremal Betti number for $s \gg 0$.
	\end{thm}    
	\begin{proof}
		It follows from \cite{CHT, vijay} that there exist non-negative integers $a,b$ and $ s_0$ such that $\reg(R/I^s) =as+b$ for all $ s \geq s_0$. For each $s \geq s_0$, there exist $i_s $ such that $\beta_{i_s,i_s+as+b}(R/I^s)$ is an extremal Betti number. Fix $ m \geq 2$. Then, by Corollary \ref{reg}, $$\reg(R/(I^s)^{[m]})=m(as+b)+(m-1)i_s \; \text{ for } s \geq s_0.$$ Now, $I^{[m]}$ is a homogeneous ideal. Therefore, it follows from \cite{CHT, vijay} that there exist non-negative integers $a_m,b_m$ and $s_m$ such that $\reg(R/(I^{[m]})^s) =a_ms+b_m$ for all $ s \geq s_m$. By Lemma \ref{tech-lem11},  $\reg(R/(I^{[m]})^s) =\reg(R/(I^s)^{[m]}).$ Thus,   $a_ms+b_m=m(as+b)+(m-1)i_s$ for $s \geq \max\{s_0,s_m\}$. By comparing, we get $a_m=am$ and $b_m=mb+(m-1)i_s$. Since $b, b_m$ are constant for $ s \geq \max\{s_0,s_m\}$, $i_s$ is a constant for $ s \geq \max\{s_0,s_m\}$. Hence, the assertion follows.
	\end{proof}
	\section{Square Powers of Monomial Ideals}
	In this section, we consider monomial ideals. First note that if $\cha(\K) =p >0$, then for a monomial ideal $I$, the $p^e$-th square power and the $p^e$-th Frobenius power are same.  
	We now prove an auxiliary result about square powers of monomial ideals.
	\begin{lem}\label{tech-lem0}
		Let $I$ and $J$ be monomial ideals of $R $. 
		Then, $(I\cap J ) ^{[m]} = I^{[m]} \cap J^{[m]}$, for $m\geq 2$.
	\end{lem}
	\begin{proof}
		Let $I$ be minimally generated by the set $\lbrace X^{a_1} ,\hdots , X^{a_k}  \rbrace $, where $a_1, \hdots , a_k \in \mathbb{N}^n$. Then,  for $m\geq 1 $, $I^{[m]}=\langle X^{ma_1} , \hdots , X^{ma_k} \rangle$.
		Let $J$ be minimally generated by  the set $\{ X^{b_1}, \hdots , X^{b_l} \},$ where $b_1,\ldots,b_l \in \mathbb{N}^n$. Therefore, by \cite[Proposition 1.2.1]{Her}, $I\cap J$ is minimally generated by the set $\{ \lcm(X^{a_i},X^{b_j}): 1\leq i \leq k , 1\leq j\leq l \}$. Thus,  \begin{align*}(I\cap J)^{[m]} &= \langle \lcm(X^{a_i},X^{b_j})^m: 1\leq i \leq k , 1\leq j\leq l \rangle \\&=\langle \lcm(X^{ma_i},X^{mb_j}): 1\leq i \leq k , 1\leq j\leq l \rangle \\& = I^{[m]}\cap J^{[m]},\end{align*}
		where the last equality follows from \cite[Proposition 1.2.1]{Her}. Hence, the assertion follows.
	\end{proof}
	Now, we compute primary decomposition of $I^{[m]}$ for any monomial ideal $I$.
	\begin{thm}\label{main-thm1}
		Let $I$ be a monomial ideal with primary decomposition $I= \bigcap\limits_{i=1}^r Q_i$. Then,  $I^{[m]}= \bigcap\limits_{i=1}^r Q_i^{[m]}$ is primary decomposition of $I^{[m]}$ for all $m\geq 2$. Moreover, ${\ass}(I)={\ass}(I^{[m]})$ and ${\Min}(I)={\Min}(I^{[m]})$, for all $m \geq 2$.
	\end{thm}
	\begin{proof}
		We prove this assertion by induction on $r$. If $r=1$, then by \cite[Theorem 1.3.1]{Her}, $I$ is generated by pure powers of the variables, which implies that $I^{[m]}$ is also generated by pure powers of the variables. By \cite[Theorem 1.3.1]{Her},  the assertion follows. Assume that $r \geq 2$ and the result is true for $r-1$. Set $J=\bigcap\limits_{i=1}^{r-1} Q_i$. Note that $I =J \cap Q_r$. By Lemma \ref{tech-lem0}, $I^{[m]}=J^{[m]} \cap Q_r^{[m]}$ for any $m \geq 2$.  By induction, $J^{[m]}= \bigcap\limits_{i=1}^{r-1} Q_i^{[m]}$ is primary decomposition of $J^{[m]}$ for all $m\geq 2$. Therefore, $I^{[m]}= \bigcap\limits_{i=1}^r Q_i^{[m]}$. Since $Q_i^{[m]}$ is generated by pure powers of variables, $ \bigcap\limits_{i=1}^r Q_i^{[m]}$ is primary decomposition of $I^{[m]}$, by \cite[Theorem 1.3.1]{Her}.
	\end{proof}
	The following example shows that $\ass(I)$ and $\ass(I^{[m]})$ need not be the same in general.
	\begin{exmp}\label{count ass}
		Let $I= \langle X-Y , X-Z \rangle \subset \mathbb{K}[X,Y,Z]$. Observe that $I$ is a prime ideal. If $\cha(\mathbb{K}) \neq 2$, then the associated primes of $I^{[2]} = \langle X^2 -Y^2 , X^2 - Z^2 \rangle $  are $\langle -X+Z ,-X+Y \rangle , \langle - X+ Z , X+Y  \rangle , \langle X+Z , -X +Y \rangle , \langle X+Z , X+Y \rangle $.
	\end{exmp}

	An ideal $I$ in  $R$ is called \textit{normally torsion-free} if $\ass({I^s})\subseteq \ass({I})$ for all $s \geq 1$. Given a normally torsion-free monomial ideal, one can construct a family of normally torsion-free monomial ideals in the following way:
	
	\begin{thm}\label{normally torsion free}
		Let $I \subset R$ be a monomial ideal. If $I$ is normally torsion-free, then $I^{[m]}$ is normally torsion-free, for all $m \geq 2$. 
	\end{thm}
	\begin{proof}
		Since $I$ is a normally torsion-free ideal, $\ass({I^s})\subseteq \ass({I})$  for all $s$. Let $ s \geq 1$ and $ m\geq 2$.  By Lemma \ref{tech-lem11}, $(I^{[m]})^s= (I^s)^{[m]}$. Now, by Theorem \ref{main-thm1}, $\ass({(I^{[m]})^s})= \ass({(I^s)^{[m]}})= \ass ( {I^s})\subseteq \ass ({I})= \ass ({I^{[m]}})$. Hence,  $I^{[m]}$ is normally torsion-free for all $m$. 
	\end{proof}
	In the following result, we compute the index of stability of $I^{[m]}$ in terms of index of stability of $I$.

	\begin{thm}\label{index}
		The index of stability of $I$ and $I^{[m]}$ are same for all $m \geq 2$.
	\end{thm}
	\begin{proof}
		Let $k \geq 2$ and $m \geq 2$. Then,   by Lemma \ref{tech-lem11}, $(I^{[m]})^k=(I^k)^{[m]}$. Therefore, by Theorem \ref{main-thm1}, $\ass({(I^{[m]})^k})= \ass({(I^k)^{[m]}})= \ass ( {I^k})$. Hence, the assertion follows.
	\end{proof}
	Given an ideal $I \subset R$, its $s$-th symbolic power of $I$ is defined as
	\begin{center}
		$$I^{(s)}=\bigcap\limits _{\p\in \Min (I)} (I^s R_{\p} \cap R), $$
	\end{center}
	where $R_{\p}$ is the localisation of $R$ at $\p$.

	Next, we study the depth function $\depth (R/I^{(s)})$. Nguyen and Trung proved \cite[Corollary 3.2]{TrungInv} that if $I$ is a monomial ideal, then $\depth ({R}/{I^{(s)}})$ is asymptotically periodic function i.e. it is a periodic function for $s \gg0$. They also proved that \cite[Theorem 3.3]{TrungInv} if $I$ is a monomial ideal such that $I^{(s)}$ is integrally closed for $s \gg 0$, then  $\depth (R/{I^{(s)}})$ is a convergent function, and \begin{displaymath}
		\lim_{s\to \infty} \depth (R/{I^{(s)}})= n- \dim (F_{\mathfrak{s}}(I)),
	\end{displaymath} where \begin{center}
		$F_{\mathfrak{s}}(I) := \bigoplus_{s\geq0}{I^{(s)}}/{\m I^{(s)}}$
	\end{center}
	and $\m$ is the homogeneous maximal ideal of $R$. Also, they presented a large class of monomial ideals $I$  with the property that $I^{(s)}$ is integrally closed for $s \gg 0$, \cite[Lemma 3.5]{TrungInv}. Here, we give a class of monomial ideals $I$ with the property that $I^{(s)}$ is not integrally closed for all $s \geq 2$ and  $\depth (R/{I^{(s)}})$ is a convergent function, and \begin{displaymath}
		\lim_{s\to \infty} \depth (R/{I^{(s)}})= n - \dim (F_{\mathfrak{s}}(I)).
	\end{displaymath}
	
	First, we prove the following lemma showing that the integral closure of the ordinary power and that of the square power is same for a monomial ideal.    
	\begin{lem}\label{integral}
		Let $I\subset R$ be a monomial ideal. Then,  for $m \geq 1$
		$$\overline{I^{[m]}}= \overline{I^m}.$$
	\end{lem}
	\begin{proof}
		Since $I^{[m]} \subset I^m$,  $\overline{I^{[m]}}\subset \overline{I^m}$. Let $u \in I^m$ be a monomial. By Lemma \ref{tech-lem11}, $ (I^m)^{[m]}=(I^{[m]})^m $ which implies that $u^m \in (I^{[m]})^m$, and hence, by \cite[Theorem 1.4.2]{Her}, $u  \in \overline{I^{[m]}}$. Therefore, $I^m \subset \overline{I^{[m]}}$. Hence, the assertion follows.
	\end{proof}
	It is clear from Lemma \ref{integral} that if $I$ is a monomial ideal, then $I^{[m]}$ is never integrally closed for $ m \geq 2$.
	\begin{lem}\label{tech-lem1}
		Let $I\subset R$ be a monomial ideal. Then,  for all $m\geq 2 $ and $s\geq 2 $,
		$$(I^{[m]})^{(s)}= (I^{(s)})^{[m]}.$$
	\end{lem}
	\begin{proof}  The assertion is immediate from Theorem \ref{main-thm1} and the definition of the symbolic power of ideals.
	\end{proof}
	We conclude this article with the following result:
	\begin{thm}\label{main-thm3}
		Let $I\subset R$ be a monomial ideal such that $I^{(s)}$ is integrally closed for $s \gg 0$. Then,   the sequence $\{\depth(R/(I^{[m]})^{(s)})\}_{s\geq1}$ is a convergent sequence for any $m \geq 2$, and
		\[ \lim_{s \rightarrow \infty} \depth(R/(I^{[m]})^{(s)})= n- \dim(F_{\mathfrak{s}}(I^{[m]})).\]
	\end{thm}
	\begin{proof}
		It follows from Lemma \ref{tech-lem1} that for any $s \geq 1$, $(I^{[m]})^{(s)}= (I^{(s)})^{[m]}$. Therefore,   $\depth(R/(I^{[m]})^{(s)})=\depth(R/(I^{(s)})^{[m]})$. Now, by Corollary \ref{cor-depth}, \[\depth(R/(I^{[m]})^{(s)})=\depth(R/I^{(s)}).\] Therefore,  \[ \lim_{s \rightarrow \infty} \depth(R/(I^{[m]})^{(s)})= \lim_{s \rightarrow \infty} \depth(R/I^{(s)})=n- \dim(F_{\mathfrak{s}}(I)),\] where the last equality follows from \cite[Theorem 3.3]{TrungInv}. Now, it is enough to prove that $\dim(F_{\mathfrak{s}}(I^{[m]}))=\dim(F_{\mathfrak{s}}(I))$. Note that for each $s \geq 0$, $$\ell\left({I^{(s)}}/{\mathfrak{m}I^{(s)}}\right) =\mu(I^{(s)}) =\mu((I^{(s)})^{[m]})=\mu((I^{[m]})^{(s)}) =\ell\left({(I^{[m]})^{(s)}}/{\mathfrak{m}(I^{[m]})^{(s)}}\right).$$ Hence, the assertion follows.
	\end{proof}
	
	\section*{Acknowledgement}
	The authors are grateful to Prof. Sarang S. Sane, who initially suggested looking at the behavior of the projective dimension of square powers of monomial ideals from where this article grew up. The authors are thankful to Prof. A. V. Jayanthan for his valuable inputs on several drafts and fruitful conversations in respect of making this article. The authors are also thankful to Prof. Rafael Villarreal for many helpful comments and suggestions. The authors feel privileged and fortunate for getting a research environment at the Indian Institute of Technology Madras.
	
\end{document}